\documentclass[a4paper,oneside,10pt]{amsart}

\usepackage{a4wide}
\usepackage[T1]{fontenc}
\usepackage[ansinew]{inputenc}
\usepackage{lmodern} 
\usepackage{graphicx}
\usepackage{amsmath}
\allowdisplaybreaks
\usepackage{amsthm}
\usepackage{amsfonts}
\usepackage{amssymb}
\usepackage{setspace}
\usepackage{mathrsfs}
\usepackage[all]{xy}
\usepackage{enumerate}
\usepackage{xcolor}
\usepackage{catchfile}
\usepackage{autobreak}

\theoremstyle{plain}

\newtheorem{theorem}{Theorem}[section]

\newtheorem{proposition}[theorem]{Proposition}
 \newtheorem{lemma}[theorem]{Lemma}

\theoremstyle{definition}
\newtheorem{remark}[theorem]{Remark}

\newtheorem{assumption}[theorem]{Assumption}

 \newtheorem{definition}[theorem]{Definition}

\newtheorem*{proposition*}{Proposition}
\newtheorem*{definition*}{Definition}
\numberwithin{equation}{section}
\theoremstyle{plain}
\newtheorem*{theorem*}{Theorem}

\newenvironment{abc}{\begin{enumerate}[{\rm (a)}]}{\end{enumerate}}
\newenvironment{num}{\begin{enumerate}[{\rm 1.}]}{\end{enumerate}}
\newenvironment{iiv}{\begin{enumerate}[{\rm (i)}]}{\end{enumerate}}

\def\dom{\mathrm{D}}
\def\sg#1{(#1(t))_{t\geq0}}
\def\semis{\mathscr{P}}
\def\dd{\mathrm{d}}
\def\res{\mathrm{R}}

\def\ee{\mathrm{e}}
\def\RR{\mathbb{R}}
\def\AL{\mathrm{AL}}
\def\NN{\mathbb{N}} 
\def\LLL{\mathscr{L}}
\def\M{\mathrm{M}}

\def\BBB{\mathscr{B}}

\def\BC{\mathrm{C}_{\mathrm{b}}}
\def\BC{\mathrm{C}_{\mathrm{b}}}
\def\tlim{\mathop{\tau\lim}}
\def\tlimi{\mathop{\tau\liminf}}

\makeatletter
\newenvironment{proofof}[1]{\par
	\pushQED{\qed}%
	\normalfont \topsep6\p@\@plus6\p@\relax
	\trivlist
	\item[\hskip\labelsep
				\bfseries
		Proof of #1\@addpunct{.}]\ignorespaces		
}{%
	\popQED\endtrivlist\@endpefalse
}
\makeatother

\begin{document}
\title{Positive Miyadera--Voigt perturbations of bi-continuous semigroups} 

\author{Christian Budde}
\address{North-West University, School of Mathematical and Statistical Sciences, Potchefstroom Campus, Private Bag X6001, Potchefstroom 2520, South Africa}
\email{Christian.Budde@nwu.ac.za}


\begin{abstract}                                                                         
We discuss positive Miyadera--Voigt type perturbations for bi-continuous semigroups on $\AL$-spaces with an additional locally convex topology generated by additive seminorms. Our main example is the space of bounded Borel measures.
\end{abstract}

\thanks{The author was supported by the DAAD-TKA Project 308019 ``\emph{Coupled systems and innovative time integrators}''.}

\date{}
\maketitle

\section*{Introduction}

\noindent Various models of physical processes ask for positive solutions in order to have a reasonable interpretation, e.g., consider solutions containing the absolute temperature or a density. The maximum principle for elliptic and parabolic partial differential equations guarantees positive solutions under positive initial data. This demonstrates the importance of positivity in the theory of operator semigroups on Banach spaces, which in fact appear as solutions of a special class of PDE's, the so called evolution equations which in turn can be rewritten as abstract Cauchy problems on a Banach space. The fundamental concepts in this context such as vector lattices, Banach lattices and positive operators are studied in detail for example in \cite{Schaefer1974} and \cite{PMN1991}. Positive operator semigroups are treated for example by Arendt et al. \cite{Positive1986} and more recently by B\'{a}tkai, Kramar-Fijav\v{z} and Rhandi \cite{Positive2017}. Perturbations of $C_0$-semigroup are discussed by several authors for example Engel and Nagel \cite[Chapter III, Sect.~3]{EN}, Voigt \cite{Voigt1977}, Desch and Schappacher \cite{DS1984}, Bombieri \cite{BombieriPhD} and Adler, Bombieri and Engel in \cite{ABE2014}. 

\medskip
\noindent Markov processes associated to stochastic differential equations or jointly continuous flows on metric spaces give rise to semigroups which are in general not strongly continuous with respect to the norm of the Banach space the semigroup is working on, cf. \cite[Sect.~2.5]{LB} and \cite[Sect.~3.2]{KuPHD}, but they enjoy strong continuity with respect to a weaker additional locally convex topology on the Banach space. The theory of bi-continuous semigroups was introduced by K\"uhnemund \cite{KuPHD} and was further studied by Farkas \cite{FaPhD}. In this setting there are perturbation results by Farkas \cite{FaSF,FaStud} and recentely by Budde and Farkas \cite{BF2}. Positivity in this context makes an appearance in \cite{ESF2005}.

\medskip This paper is inspired by an article by Voigt \cite{Voigt1989}, where positive operator semigroups and perturbations are combined and the following perturbation result for positive $C_0$-semigroups was proved.

\begin{theorem*}{\cite[Thm.~0.1]{Voigt1989}}
Let $E$ be an $\AL$-space, and let $A$ be the generator of a positive $C_0$-semigroup on $E$. Let $B:\dom(A)\rightarrow E$ be a positive operator, and assume that $A+B$ is resolvent positive. Then $A+B$ is the generator of a positive $C_0$-semigroup.
\end{theorem*}

Quite a number of other positive perturbation results for strongly continuous semigroups and their applications are handled by Arlotti and Banasiak \cite{PertPosAppl2006}. In this paper we consider positive perturbations of bi-continuous semigroups in the style of Voigt's work. In particular, we use the Miyadera--Voigt perturbation theorem for bi-continuous semigroups proved by Farkas \cite{FaPhD}.

\medskip\noindent The paper is organized as follows: in Section 1 we recall the Miyadera--Voigt perturbation theorem for bi-continuous semigroups and state Theorem \ref{thm:BiMVPos} as our main result, whose proof is contained in Section 2. In the last section we discuss rank-one perturbations and bi-continuous semigroups on the space $\M(\Omega)$ of bounded Borel measures in connection with differentiable measures. We remark that this paper is based on the authors's PhD thesis \cite[Chapter 5]{BuddePhD}. 

\section{Preliminaries}

\subsection{Bi-Continuous Semigroups}
The class of bi-continuous semigroups was introduced by K\"uhnemund in \cite{Ku}. Given a Banach space $X$, The idea is to add a locally convex topology $\tau$ on satisfying the following technical assumptions.

\begin{assumption}\label{asp:bicontspace} 
Consider a triple $(X,\|\cdot\|,\tau)$, where $X$ is a Banach space, and
\begin{num}
\item $\tau$ is a locally convex Hausdorff topology coarser than the norm-topology on $X$, i.e. the identity map $(X,\|\cdot\|)\to(X,\tau)$ is continuous;
\item $\tau$ is sequentially complete on the $\left\|\cdot\right\|$-closed unit ball;
\item The dual space of $(X,\tau)$ is norming for $X$, i.e.,
\begin{equation}\label{eq:norm}
\|x\|=\sup_{\substack{\varphi\in(X,\tau)'\\\|\varphi\|\leq1}}{|\varphi(x)|},\quad x\in X.\end{equation}
\end{num}
\end{assumption}

Notice that every locally convex topology $\tau$ yields a family of continuous seminorms $\semis$ and vice versa, cf. \cite[Chapter II, Sect.~4]{Schaefer1971} or \cite[Thm.~1.36 \& 1.37]{Rudin}. In particular, if we want to make explicit calculations within the framework of locally convex topologies, we use the corresponding seminorms.

\begin{remark}\label{rem:SumSeminorm}
Since a locally convex topology, generated by a family of seminorms $\semis$, does not change if we add another continuous seminorm to $\semis$, we can assume that $\semis$ also contains all finite positive linear combinations of their seminorms, see also \cite[Prop.~7.1.4]{Conway}. More observations on Assumptions \ref{asp:bicontspace} and the relations to Saks spaces can be found in \cite[Rem.~5.2]{BF}.
\end{remark}

Now we give the definition of a bi-continuous semigroup.

\begin{definition}[K\"uhnemund \cite{Ku}]\label{def:bicontsemi}
Let $X$ be a Banach space with norm $\|\cdot\|$ together with a locally convex topology $\tau$ such that the conditions in Assumption \ref{asp:bicontspace} are satisfied. We call $(T(t))_{t\geq0}$ a \emph{bi-continuous semigroup} if
\begin{num}
\item $ T(t+s)=T(t)T(s)$ and $T(0)=I$ for all $s,t\geq 0$,
\item $(T(t))_{t\geq0}$ is strongly $\tau$-continuous, i.e. the map $\varphi_x:[0,\infty)\to(X,\tau)$ defined by $\varphi_x(t)=T(t)x$ is continuous for every $x\in X$,
\item $(T(t))_{t\geq0}$ has type $(M,\omega)$ for some $M\geq 1$ and $\omega\in \RR$, i.e., $\left\|T(t)\right\|\leq M\ee^{\omega t}$ for all $t\geq0$,
\item $(T(t))_{t\geq0}$ is locally-bi-equicontinuous, i.e., if $(x_n)_{n\in\NN}$ is a norm-bounded sequence in $X$ which is $\tau$-convergent to $0$, then also $(T(s)x_n)_{n\in\NN}$ is $\tau$-convergent to $0$ uniformly for $s\in[0,t_0]$ for each fixed $t_0\geq0$.
\end{num}
\end{definition}

Significant examples in this context are evolution semigroups on $\BC(\RR,X)$ \cite{Sci2016}, Koopman semigroups \cite{Kuehner2019},\cite[Chapter 4]{EFHN2015}, semigroups induced by flows \cite[Sect.~3.2]{KuPHD}, adjoint semigroups \cite[Sect.~3.5]{KuPHD} and the Ornstein--Uhlenbeck semigroup on $\BC(\RR^d)$ \cite[Sect.~2.3]{FaPhD},\cite{FaSF},\cite{FL2009}, to mention a few.

\medskip
As in the case of $C_0$-semigroups, we can define the generator of a bi-continuous semigroup.

\begin{definition}Let $(T(t))_{t\geq0}$ be a bi-continuous semigroup on $X$. The \emph{generator} $A$ is defined by
\[Ax:=\tlim_{t\to0}{\frac{T(t)x-x}{t}}\] with the domain
\[\dom(A):=\Bigl\{x\in X:\ \tlim_{t\to0}{\frac{T(t)x-x}{t}}\ \text{exists and} \ \sup_{t\in(0,1]}{\frac{\|T(t)x-x\|}{t}}<\infty\Bigr\}.\]
\end{definition}

This generator has a number of important properties which are summarized in the following theorem (see \cite{Ku}, \cite{FaStud}):

\begin{theorem}\label{thm:BiProp}
Let $(T(t))_{t\geq0}$ be a bi-continuous semigroup with generator $A$. Then the following hold:
\begin{abc}
\item The operator $A$ is bi-closed, i.e., whenever $x_n\stackrel{\tau}{\to}x$ and $Ax_n\stackrel{\tau}{\to}y$ and both sequences are norm-bounded, then $x\in\dom(A)$ and $Ax=y$.
\item The domain $\dom(A)$ is bi-dense in $X$, i.e., for each $x\in X$ there exists a norm-bounded sequence $(x_n)_{n\in\NN}$ in $\dom(A)$ such that $x_n\stackrel{\tau}{\to}x$.
\item For $x\in\dom(A)$ we have $T(t)x\in\dom(A)$ and $T(t)Ax=AT(t)x$ for all $t\geq0$.
\item For $t>0$ and $x\in X$ one has \begin{align}\int_0^t{T(s)x\ \dd s}\in\dom(A)\ \ \text{and}\ \ A\int_0^t{T(s)x\ \dd s}=T(t)x-x. \end{align}
\item For $\lambda>\omega$ one has $\lambda\in\rho(A)$ (thus $A$ is closed) and: \begin{align}\label{eq:bicontlaplace}
\res(\lambda,A)x=\int_0^{\infty}{\ee^{-\lambda s}T(s)x\ \dd s},\quad x\in X\end{align} where the integral is a $\tau$-improper integral.
\end{abc}
\end{theorem}

According to Theorem \ref{thm:BiProp}$\mathrm{(b)}$ we make the following additional definition, coming from \cite{FaPhD}, which will be of importance later.

\begin{definition}
For $\eta>1$, we call a subset $D\subseteq X$ \emph{$\eta$-bi-dense} in $X$ for $\tau$ if for all $x\in X$ there exists a sequence $(x_n)_{n\in\NN}$ in $D$ which converges to $x$ with respect to $\tau$ and furthermore,
\begin{align*}\tag{$\eta$D}
\left\|x_n\right\|\leq\eta\left\|x\right\|,\quad\text{for}\ n\in\NN.
\end{align*}
\end{definition}

From \cite[Prop.~1.7]{FaStud} we recall that if $(A,\dom(A))$ is a generator of a bi-continuous semigroup $(T(t))_{t\geq0}$, then $A$ is automatically $\eta$-bi-dense for some $\eta>1$. We also need the following notion. 

\begin{definition}
A bounded operator $B\in\LLL(X)$ is called \emph{local} if for each $\varepsilon>0$ and each $p\in\semis$ there exist a constant $K>0$ and $q\in\semis$ such that
\[
p(Bx)\leq Kq(x)+\varepsilon\left\|x\right\|,\quad x\in X.
\]
\end{definition}

\begin{definition}
A bounded function $F:\left[0,t_0\right]\to\LLL(X)$ is called \emph{local} if for all $p\in\semis$ and $\varepsilon>0$ there exists $K>0$ and $q\in\semis$ such that for all $t\in\left[0,t_0\right]$ and $x\in X$
\[
p(F(t)x)\leq Kq(x)+\varepsilon\left\|x\right\|.
\]
\end{definition}

\begin{remark}\label{rem:localtight}
\begin{abc}
	\item It was shown in \cite{FaPhD},\cite{FaSF},\cite{FaStud} that for $(A,\dom(A))$ the generator of a bi-continuous semigroup and $b\geq a\geq\omega$ the resolvent family $\left\{\res(\lambda,A):\ \lambda\in\left[a,b\right]\right\}$ is local.
	\item Every bi-continuous semigroup $(T(t))_{t\geq0}$ on $\BC(\Omega)$, for $\Omega$ a Polish space, the set\newline$\left\{T(t):\ t\in\left[0,t_0\right]\right\}$ is local for each $t_0>0$.
	\item  Let $(T(t))_{t\geq0}$ be a bi-continuous semigroup on $\BC(\Omega)$. Moreover, if $\mathscr{K}\subseteq\M(\Omega)$ is a norm-bounded and weak$^*$-compact subset of the space of bounded Borel measures, then the set $\left\{T'(t)\nu:\ \nu\in\mathscr{K}\right\}$ is tight/local. As example one can take $\mathscr{K}=\M_1(\Omega)$ if $\Omega$ is compact, cf. \cite{Farkas2011}.
\end{abc}	
\end{remark}

\subsection{Adjoints of bi-continuous semigroups}\label{subsec:AdBiCont}

We recall some essential results on adjoint of bi-continuous semigroups from \cite{Farkas2011}, since we need these later on in Section \ref{subsec:GWMeasure}. Let $X$ be a Banach space and let $\tau$ be a locally convex topology satisfying Assumption \ref{asp:bicontspace}. Furthermore let $(T(t))_{t\geq0}$ be a bi-continuous semigroup on $X$ with respect to $\tau$. As in the case of adjoints of strongly continuous semigroups, where one consideres the sun dual $X^{\odot}$ in order to have again a strongly continuous semigroup, cf. \cite[Sect.~1.3]{van1992adjoint} or \cite[Chapter II, Sect.~2.6]{EN}, we now examine the subspace $X^{\circ}$ of the dual space $X'$ consisting of all norm-bounded linear functionals which are $\tau$-sequentially continuous on norm-bounded sets of $X$. As a matter of fact $X^{\circ}$ is a closed linear subspace of $X'$ and hence a Banach space if equipped with the inherited norm of $X'$. Furthermore, $X^{\circ}$ can be equipped with the topology $\tau^{\circ}:=\sigma(X^{\circ},X)$. In order to show that $\tau^{\circ}$ satisfies Assumption \ref{asp:bicontspace}, we have to postulate that $X^{\circ}\cap\overline{B(0,1)}$ is sequentially complete with respect to $\sigma(X^{\circ},X)$. We remark that this assumption in general does not follow from the general Assumptions \ref{asp:bicontspace}.

\begin{proposition}{\cite[Prop.~2.3]{Farkas2011}}
Let $B\in\LLL(X)$ be a norm-bounded linear operator which is $\tau$-sequentially continuous on norm-bounded sets. Then the adjoint $B'\in\LLL(X')$ leaves $X^{\circ}$ invariant.
\end{proposition}

From the previous result we conclude that we can restrict $(T'(t))_{t\geq0}$ to the space $X^{\circ}$. We denote this restricted semigroup by $(T^{\circ}(t))_{t\geq0}$. To conclude that $(T^{\circ}(t))_{t\geq0}$ is bi-continuous on $X^{\circ}$ with respect to $\tau^{\circ}$ we again have to impose a additional hypothesis. Especially we have to assume that every norm-bounded $\tau^{\circ}$-null sequence $(\varphi_n)_{n\in\NN}$ in $X^{\circ}$ is $\tau$-equicontinuous on norm bounded sets.

\medskip
Let us continue with some examples. Actually, we consider semigroups $(T(t))_{t\geq0}$ on the Banach space $X:=\BC(\Omega)$, where $\Omega$ is a Polish space, which are bi-continuous with respect to the compact-open topology $\tau_{\mathrm{co}}$. In \cite[Sect.~3]{Farkas2011} Farkas illustrated that $\BC(\Omega)^{\circ}$ coincides with $\M(\Omega)$, the space of bounded Borel measures. The following results connect bi-continuous semigroups on $\BC(\Omega)$ with these on $\M(\Omega)$.

\begin{theorem}{\cite[Thm.~3.5]{Farkas2011}}\label{thm:DualCbM}
Let $\Omega$ be a Polish space and $(T(t))_{t\geq0}$ bi-continuous on $\BC(\RR)$ with respect to $\tau_{\mathrm{co}}$. Then the semigroup $(T^{\circ}(t))_{t\geq0}$ defined as $T^{\circ}(t):=T'(t)_{|\M(\Omega)}$, $t\geq0$, is a bi-continuous semigroup on $\M(\Omega)$ with respect to $\tau^{\circ}$.
\end{theorem}

Counter-intuitively the converse of the previous theorem also holds true.

\begin{theorem}{\cite[Thm.~3.6]{Farkas2011}}
Let $\Omega$ be a Polish space. Let $(S(t))_{t\geq0}$ be a bi-continuous semigroup on $\M(\Omega)$ with respect to $\tau^{\circ}$. Then there exists a semigroup $(T(t))_{t\geq0}$ on $\BC(\RR)$ which is bi-continuous with respect to $\tau_{\mathrm{co}}$ and such that $T^\circ(t)=S(t)$ for all $t\geq0$
\end{theorem}

The following result yields a characterization for the generator of adjoints of bi-continuous semigroups by means of adjoints of unbounded operators. It is strongly related to \cite[Chapter I, Sect.~2.5]{EN} and \cite[Thm.~1.2.3]{van1992adjoint}. The proof of Lemma \ref{lem:DomAdBiContSemi} is essential the same and so we omit it here.

\begin{lemma}\label{lem:DomAdBiContSemi}
Let $(T(t))_{t\geq0}$ be bi-continuous on $X$ with respect to $\tau$ and assume that the additional hypotheses on $X^{\circ}$ and $\tau^{\circ}$ from above hold. Let us denote the generator of $(T^{\circ}(t))_{t\geq0}$ by $(A^{\circ},\dom(A^{\circ}))$. Then
\[
\dom(A^{\circ})=\left\{x'\in X^{\circ}:\ \exists y'\in X^{\circ}\ \forall x\in\dom(A):\ \left\langle Ax,x'\right\rangle=\left\langle x,y'\right\rangle\right\},\quad A^{\circ}x'=y'.
\]
\end{lemma}

\subsection{Positivity and Bi-$\AL$-spaces}

In \cite{FaStud} Farkas considered bounded and in \cite{FaSF} Miyadera--Voigt perturbation of bi-continuous semigroups, and more recently, Budde and Farkas also studied Desch--Schappacher type perturbations, see \cite{BF2}. In this paper we take positivity of Miyadera--Voigt perturbations into account as Voigt in \cite{Voigt1989} did for strongly continuous operator semigroups on Banach spaces. In \cite{ArendtRhandi1991} Arendt and Rhandi characterize positive perturbations by multiplication operators and apply this to elliptic Schr\"odinger operator. We recall the following definitions in order to set up the theory.

\begin{definition}
A \emph{vector lattice} or \emph{Riesz space} is a vector space $V$ equipped with a partial order $\leq$ such that for each $x,y,z\in V$:
\begin{abc}
	\item $x\leq y \Rightarrow x+z\leq y+z$.
	\item $x\leq y \Rightarrow \alpha x\leq\alpha y$ for all scalars $\alpha\geq0$.
	\item For any pair $x,y\in V$ there exists a supremum, denoted by $x\vee y$, and a infimum, denoted by $x\wedge y$, in $V$ with respect to the partial order $\leq$.
\end{abc}
An element $x\in V$ is called \emph{positive} if $x\geq0$. The set of all positive elements of $V$ is denoted by $V_+$. Furthermore, the \emph{absolute value} of an element $x\in V$ is defined by $\left|x\right|:=x\wedge(-x)$.
\end{definition}

\begin{definition}
A \emph{Banach lattice} is a Banach space $(X,\left\|\cdot\right\|)$ which is a Riesz space with an partial order such that for all $x,y\in X$: $\left|x\right|\leq\left|y\right| \Rightarrow \left\|x\right\|\leq\left\|y\right\|$.
\end{definition}

Let $X$ be a Banach lattice and $T\in\LLL(X)$. Then $T$ is called positive, denoted by $T\geq0$, if $Tx\geq0$ for each $x\in X_+$. A semigroup of bounded linear operators $(T(t))_{t\geq0}$ on such a Banach lattice is called positive if $T(t)\geq0$ for each $t\geq0$. In connection with unbounded operators, the following definition which was suggested by Arendt \cite{Arendt1987}.

\begin{definition}
A linear operator $(A,\dom(A))$ on a Banach lattice $X$ is called \emph{resolvent positive} if there exists $\omega\in\RR$ such that $(\omega,\infty)\subseteq\rho(A)$ and such that $R(\lambda,A)\geq0$ for each $\lambda>\omega$.
\end{definition}

In the theorem of Voigt mentioned in the introduction the concept of so-called $\AL$-spaces is significant. These spaces satisfy a special kind of norm property, cf. \cite[Chapter II, Sect.~8]{Schaefer1974}. As we have to take care of an additional locally convex topology in the setting of bi-continuous semigroups, we introduce the following related notion.

\begin{definition}\label{def:biAL}
Let $(X,\left\|\cdot\right\|,\leq)$ be a Banach lattice with ordering and locally convex topology $\tau$ generated by a family $\semis$ of seminorms which satisfies the Assumptions \ref{asp:bicontspace}. We say that $X$ is a \emph{bi-$\AL$ space} if it is a $\AL$-space, i.e., for all $x,y\in X_+$ the equality
\[
\left\|x+y\right\|=\left\|x\right\|+\left\|y\right\|
\]
holds, and there exists $\semis_+\subseteq\semis$ such that $\semis_+$ still generates the locally convex topology $\tau$ and for all $x,y\in X_+$
\[
p(x+y)=p(x)+p(y)
\]
for each $p\in\semis_+$.
\end{definition}

In Section \ref{subsec:GWMeasure} we will consider an explicit example of a space which satisfies the properties of Definition \ref{def:biAL}. Locally convex spaces satisfying such additivity conditions as in Definition \ref{def:biAL} are also mentioned in \cite{Cristescu2008} to discuss regular operators on vector lattices. 


\subsection{Positive perturbations}

The main result of this paper is the following.

\begin{theorem}\label{thm:BiMVPos}
Let $(A,\dom(A))$ be the generator of a positive, local bi-continuous semigroup $\sg{T}$ on a bi-$\AL$ space $X$ with $\eta$-bi-dense domain $\dom(A)$ for some $\eta>1$. Let $B:\dom(A)\rightarrow X$ be a positive operator, i.e., $Bx\geq0$ for each $x\in\dom(A)\cap X_+$, and assume that $B\res(\lambda,A)$ is local and $(A+B,\dom(A))$ is resolvent positive. Then $(A+B,\dom(A))$ is the generator of a positive bi-continuous semigroup.
\end{theorem}

As compared with \cite[Thm.~0.1]{Voigt1989} we need some additional technical assumptions in Theorem \ref{thm:BiMVPos} due to the Miyadera--Voigt perturbation theorem for bi-continuous semigroups, which we will recall here.

\begin{theorem}{\cite[Thm.~3.2.3]{FaPhD}}\label{thm:BiMV}
Let $(T(t))_{t\geq0}$ be a bi-continuous semigroup on $X$ with respect to $\tau$ with generator $(A,\dom(A))$. Suppose that $\dom(A)$ is $\eta$-bi-dense in $X_0$ and that $B:(\dom(A),\tau_A)\rightarrow(X_0,\tau)$ is continuous on $\left\|\cdot\right\|_A$-bounded sets. Suppose that there exists $t_0>0$ and $0<K<\frac{1}{\eta}$ such that
\begin{iiv}
	\item The map $s\mapsto\left\|BT(s)x\right\|$ is bounded on $\left[0,t_0\right]$ and for each $x\in\dom(A)$.
	\item $\displaystyle{\int_0^t{\left\|BT(s)x\right\|\ \dd{s}}}<K\left\|x\right\|$ for each $t\in\left[0,t_0\right]$ and $x\in\dom(A)$.
	\item For all $\varepsilon>0$ and $p\in\semis$ there exists $q\in\semis$ and $M>0$ such that
	\[
	\int_0^{t_0}{p(BT(s)x)\ \dd{s}}<Mq(x)+\varepsilon\left\|x\right\|,
	\]
	for each $x\in\dom(A)$.
\end{iiv}
Then $(A+B,\dom(A+B))$ generates a bi-continuous semigroup $(S(t))_{t\geq0}$. Furthermore, the semigroup $(S(t))_{t\geq0}$ satisfies the variation of parameter formula
\[
T(t)x=S(t)x+\int_0^t{T(t-s)BS(s)x\ \dd{s}},\quad x\in\dom(A).
\]
\end{theorem}

\begin{remark}\label{rem:Equiv}
\begin{abc}
	\item If $(T(t))_{t\geq0}$ is a bi-continuous semigroup generated by $(A,\dom(A))$, then one can deduce from \cite[Lemma~4.15]{PertPosAppl2006} the following equivalence for $\lambda\in\RR$:
\begin{align*}
&\exists M\in\left(0,1\right)\ \forall t\in\left[0,t_0\right]\ \forall x\in\dom(A): \int_0^{t}{\left\|BT(s)x\right\|\ \dd{s}}\leq M\left\|x\right\|\\
\Longleftrightarrow\ \ & \exists M'\in\left(0,1\right)\ \forall t\in\left[0,t_0\right]\ \forall x\in\dom(A): \int_0^{t}{\left\|\ee^{-\lambda s}BT(s)x\right\|\ \dd{s}}\leq M'\left\|x\right\|
\end{align*}

	\item One also easily proves the following equivalence:
\begin{align*}
&\forall\varepsilon>0\ \forall p\in\semis\ \exists K\geq0\ \exists q\in\semis\ \forall x\in\dom(A):\ \int_0^{t_0}{p(BT(s)x)\ \dd{s}}\leq Kq(x)+\varepsilon\left\|x\right\|\\
\Longleftrightarrow\ \ & \forall\varepsilon>0\ \forall p\in\semis\ \exists K'\geq0\ \exists q'\in\semis\ \forall x\in\dom(A):\ \int_0^{t_0}{p(\ee^{-\lambda s}BT(s)x)\ \dd{s}}\leq K'q'(x)+\varepsilon\left\|x\right\|
\end{align*}
\end{abc}
\end{remark}

\section{The Proof}


Recall from \cite[Chapter II, Def.~1.12]{EN} that the spectral bound $\mathrm{s}(A)$ of an linear operator is defined by
\[
\mathrm{s}(A):=\sup\left\{\mathrm{Re}(\lambda):\ \lambda\in\sigma(A)\right\},
\]
where $\sigma(A)$ denotes the spectrum of the operator $(A,\dom(A))$. In order to proof Theorem \ref{thm:BiMVPos} we need the following lemma. 

\begin{lemma}\label{lem:Aux}
Let $\eta>1$ and let $(A,\dom(A))$ be the $\eta$-bi-densely defined generator of a positive local bi-continuous semigroup $\sg{T}$ on a bi-$\AL$ space $X$. Let $M:=\sup_{t\in\left[0,1\right]}{\left\|T(t)\right\|}<\infty$ and let $B:\dom(A)\rightarrow X$ be a positive operator such that there exists $\lambda>\mathrm{s}(A)$ such that the operator $B\res(\lambda,A)$ is local and $\left\|B\res(\lambda,A)\right\|<\frac{1}{2M}<1$. Then $A+B$ is the generator of a positive bi-continuous semigroup.
\end{lemma}

\begin{proof}
We start by establishing property $\mathrm{(ii)}$ in Theorem \ref{thm:BiMV}. By Remark \ref{rem:Equiv}$\mathrm{(a)}$ it suffices that we have the following estimate for each $x\in\dom(A)_+$
\begin{align*}
\int_0^t{\left\|B\ee^{-\lambda s}T(s)x\right\|\ \dd{s}}
&=\left\|\int_0^{t}B\ee^{-\lambda s}T(s)x\ \dd{s}\right\|\\
&=\left\|B\int_0^t{\ee^{-\lambda s}T(s)x\ \dd{s}}\right\|\\
&\leq\left\|B\res(\lambda,A)x\right\|\\
&<\frac{1}{2M}\left\|x\right\|,
\end{align*}
where the first equality is justified by the $\AL$-property of the space and the fact that if $x\in\dom(A)_+$ then it is element of the space of strong continuity (which in fact coincides with $\underline{X}_0:=\overline{\dom(A)}^{\left\|\cdot\right\|}$). The second equality follows by the fact that the operator $B$ is a continuous map from $(\dom(A),\tau_A)$ to $(X,\tau)$ which is $A$-bounded by \cite[Lemma~4.1]{PertPosAppl2006}, i.e., there exists $a,b\geq0$ such that $\left\|Bx\right\|\leq a\left\|Ax\right\|+b\left\|x\right\|$ for each $x\in\dom(A)$. The first inequality follows by the Laplace transform representation of the resolvent $\res(\lambda,A)$. As a consequence we conclude that property $\mathrm{(ii)}$ of Theorem \ref{thm:BiMV} holds.  
Now let $\varepsilon>0$ and $p\in\semis_+$ be arbitrary. Then
\begin{align*}
\int_0^{t_0}{p\left(\ee^{-\lambda s}BT(s)x\right)\ \dd{s}}&=p\left(\int_0^{t_0}{\ee^{-\lambda s}BT(s)x\ \dd{s}}\right)=p\left(B\int_0^{t_0}{\ee^{-\lambda s}T(s)x\ \dd{s}}\right)\\
&\leq p\left(B\res(\lambda,A)x\right)\leq Kq(x)+\varepsilon\left\|x\right\|.
\end{align*}
Here the first step follows by the properties of the bi-$\AL$ space. The second one follows by the same argument as before and the last inequality follows by the assumption of localness of the operator $B\res(\lambda,A)$. Hence property $\mathrm{(iii)}$ is fullfilled. Moreover also $\mathrm{(i)}$ holds since $(B,\dom(B))$ is $A$-bounded by \cite[Lemma~4.1]{PertPosAppl2006}.

\medskip
Keep in mind that we showed that the properties $\mathrm{(i)}-\mathrm{(iii)}$ of Theorem \ref{thm:BiMV} hold for $x\in\dom(A)_+$. We have to show that they hold true for each $x\in\dom(A)$. In the first place we start showing that the norm condition $\mathrm{(ii)}$ of Theorem \ref{thm:BiMV} holds from each $x\in\dom(A)$. For this purpose, let $x\in\dom(A)$ and find $x_+,x_-\in X_+$ such that $x=x_+-x_-$. Especially, take $x_+:=x\vee0$ and $x_-:=-(x\wedge0)$ and observe that $x_++x_-=\left|x\right|$. For $n\in\NN$ define
\[
x_{n,\pm}:=n\int_0^{\frac{1}{n}}{T(t)x_{\pm}\ \dd{t}}\in\dom(A),
\]
and notice that $x_{n,\pm}\stackrel{\tau}{\rightarrow}x_{\pm}$ in $X$ and $x_{n,+}-x_{n,-}\stackrel{\tau_A}{\rightarrow}x$. The sequences $(x_{n,+})_{n\in\NN}$ and $(x_{n,-})_{n\in\NN}$, and hence also $(x_{n,+}-x_{n,-})_{n\in\NN}$, are norm-bounded, i.e.,
\[
\left\|x_{n,\pm}\right\|=\left\|n\int_0^{\frac{1}{n}}{T(t)x_{\pm}\ \dd{t}}\right\|\leq n\int_0^{\frac{1}{n}}{\left\|T(t)x_{\pm}\right\|\ \dd{t}}\leq L\left\|x_{\pm}\right\|,
\]
where $L:=\sup_{t\in\left[0,\frac{1}{n}\right]}{\left\|T(t)\right\|}<\infty$. Hence we obtain
\begin{align}\label{eqn:ProofLemma}
\int_0^t{\left\|\ee^{-\lambda s}BT(s)(x_{n,+}-x_{n,-})\right\|\ \dd{s}}<\frac{1}{2M}\left(\left\|x_{n,+}\right\|+\left\|x_{n,-}\right\|\right)=\frac{1}{2M}\left\|x\right\|.
\end{align}
We have to show that the integrand converges to $\left\|\ee^{-\lambda s}BT(s)x\right\|$ in order to conclude the desired estimate. Observe that for each $\varphi\in(X,\tau)'$
\begin{align*}
\varphi\left(\int_0^t{\ee^{-\lambda s}BT(s)x\ \dd{s}}\right)&=\int_0^t{\varphi\left(\ee^{-\lambda s}BT(s)x\right)\ \dd{s}}\\
&=\int_0^t{\tlim_{n\to\infty}\varphi\left(\ee^{-\lambda s}BT(s)(x_{n,+}-x_{n,-})\right)\ \dd{s}}\\
&\leq\tlimi_{n\to\infty}\int_0^t{{\left\|\ee^{-\lambda s}BT(s)(x_{n,+}-x_{n,-})\right\|\ \dd{s}}}=\frac{1}{2M}\left\|x\right\|
\end{align*}
Due to the norming property of the local convex topology (cf. Assumption \ref{asp:bicontspace}) and the $\AL$-property we obtain by taking the supremum over all $\varphi\in(X,\tau)'$ with $\left\|\varphi\right\|\leq1$ the following inequality
\[
\int_0^t{\left\|\ee^{-\lambda s}BT(s)x\right\|\ \dd{s}}=\frac{1}{2M}\left\|x\right\|.
\]

\medskip
To show the third requirement of Theorem \ref{thm:BiMV} is valid, let $\varepsilon>0$ and $p\in\semis_+$ be arbitrary and observe that there exists $K>0$ and $q\in\semis$ such that for each $n\in\NN$
\medskip
\begin{align*}
\int_0^{t_0}{p(\ee^{-\lambda s}BT(s)(x_{n,+}-x_{n,-}))\ \dd{s}}&\leq\int_0^{t_0}{p(\ee^{-\lambda s}BT(s)x_{n,+})\ \dd{s}}+\int_0^{t_0}{p(\ee^{-\lambda s}BT(s)x_{n,-})\ \dd{s}}\\
&\leq K\left(q(x_{n,+})+q(x_{n,-})\right)+\varepsilon\left(\left\|x_{n,+}\right\|+\left\|x_{n,i}\right\|\right)\\
&\leq K\left(q(x_{n,+})+q(x_{n,-})\right)+2L\varepsilon\left\|x\right\|\\
\end{align*}
Since by construction $x_{n,+}\to x_+$ and $x_{n,-}\to x_-$ we conclude that $q(x_{n,+})+q(x_{n,-})\to q\left(\left|x\right|\right)=q(x)$. The left-hand side also converges. To see this define for a fixed $\lambda> s(A)$ a sequence $(y_n)_{n\in\NN}$ in $X$ for $n\in\NN$ by
\[
y_n:=(\lambda-A)(x_{n,+}-x_{n,-}),
\]
and set
\[
y:=\tlim_{n\to\infty}{y_n}=(\lambda-A)x.
\]
Then by localness on the operator $B\res(\lambda,A)$ and the semigroup $(T(t))_{t\geq0}$ we find $K>0$ and $q,q'\in\semis$ such that
\begingroup
\allowdisplaybreaks
\begin{align*}
&\left|p(\ee^{-\lambda s}B\res(\lambda,A)T(s)y_n)-p(\ee^{-\lambda s}B\res(\lambda,A)T(s)y)\right|\\
\leq&p(\ee^{-\lambda s}B\res(\lambda,A)T(s)y_n-\ee^{-\lambda s}B\res(\lambda,A)T(s)y)\\
\leq&\ee^{-\lambda s}p(B\res(\lambda,A)T(s)(y_n-y))\\
\leq&\ee^{-\lambda s}\left(Kq(T(s)(y_n-y))+\varepsilon\left\|T(s)(y_n-y)\right\|\right)\\
\leq&K'q(T(s)(y_n-y))+\ee^{(\omega-\lambda)s}\varepsilon\left\|y_n-y\right\|\\
\leq&K''q'(y_n-y)+\varepsilon(K'+\ee^{(\omega-\lambda)s})\left\|y_n-y\right\|\\
\leq&K''q'(y_n-y)+\varepsilon M(K'+\ee^{(\omega-\lambda)s}),
\end{align*}
\endgroup
where $K':=K\ee^{-\lambda s}$ and $K''>0$ is a product of $K'$ and a constant coming from the localness of the semigroup $(T(t))_{t\geq0}$. Moreover, $M\geq0$ is a constant arising from the exponential boundedness of the semigroup. Since $y_n\stackrel{\tau}{\rightarrow}y$ and $\varepsilon>0$ was arbitrary we see that the convergence of the integrand is uniform in $s$. Thus the integral converges and we are done.
\end{proof}

The main work for proving Theorem \ref{thm:BiMVPos} is included in the previous Lemma \ref{lem:Aux}. Now we are able to accomplish the proof of our main theorem which runs in fact parallel to the original proof of \cite[Thm.~0.1]{Voigt1989}.

\begin{proofof}{Theorem \ref{thm:BiMVPos}}
By \cite[Thm.~1.1]{Voigt1989} there exists a $\lambda\in\rho(A+B)$ such that $r(B\res(\lambda,A))<1$. Moreover
\[
\res(\lambda,A)\leq\res(\lambda,A)\sum_{n=0}^{\infty}{\left(B\res(\lambda,A)\right)^n}=\res(\lambda,A+B).
\]
Now, by taking $sB$ instead of $B$ for $s\in\left[0,1\right]$ we obtain by the above
\[
\res(\lambda,A)\leq\res(\lambda,A+sB)\leq\res(\lambda,A+B).
\]
Since $B$ is positive and $\mathrm{Ran}(\res(\lambda,A+B))=\dom(A)$, we conclude that $B\res(\lambda,A+B)\in\LLL(X)$. Therefore, also $2\eta B\res(\lambda,A+B)\in\LLL(X)$ and there exists $n\in\NN$ such that
\[
\left\|2\eta B\res(\lambda,A+B)\right\|<n.
\]
Hence
\[
\left\|\frac{1}{n}B\res(\lambda,A+sB)\right\|<\frac{1}{2\eta},
\]
for each $s\in\left[0,1\right]$. In particular, one has
\[
\left\|\frac{1}{n}B\res\left(\lambda,A+\frac{j}{n}B\right)\right\|<\frac{1}{2\eta},
\]
for $0\leq j\leq n-1$. Now we apply Lemma \ref{lem:Aux} for the perturbation $\frac{1}{n}B$ repeatedly for $A,A+\frac{1}{n}B,\ldots,A+\frac{n-1}{n}B$ and obtain the generation by $A+B$ in the last step.
\end{proofof}

\section{Examples}\label{sec:Examples}
\subsection{Rank-one perturbations}

As a first example we consider rank-one perturbations as they are treated for $C_0$-semigroups by Arendt in Rhandi \cite[Thm.~2.2]{ArendtRhandi1991}. Let $(T(t))_{t\geq0}$ be a positive bi-continuous semigroup on $X$ with respect to $\tau$ with generator $(A,\dom(A))$.
Let $\semis$ be the directed family of seminorms corresponding to $\tau$ and let $\tau_A$ be the locally convex topology on $\dom(A)$ determined by the family of seminorms $\semis_A:=\left\{p(\cdot)+q(A\cdot):\ p,q\in\semis\right\}$. For a positive $\tau_A$-continuous linear functional $\varphi:\dom(A)\to\RR$ and for $y\geq0$ in $X$ we define the rank-one perturbation $B:\dom(A)\to X$ by
\[
Bx:=\varphi(x)y,\quad x\in\dom(A).
\]
The operator $(B,\dom(A))$ satisfies the assumptions of Theorem \ref{thm:BiMVPos}, and hence it is a Miyadera--Voigt perturbation. To see this let $\varepsilon'>0$ and $p\in\mathscr{P}$ be arbitrary and observe that
\[
p\left(B\res(\lambda,A)x\right)=p\left(\varphi\left(\res(\lambda,A)x\right)y\right)=\left|\varphi\left(\res(\lambda,A)x\right)\right|p(y),\quad x\in X.
\]
Since $\varphi$ is $\tau_A$-continuous we conclude that there exists $M>0$ and $p',q'\in\semis$ such that
\[
\left|\varphi(\res(\lambda,A)x)\right|\leq M\left(p'(\res(\lambda,A)x)+q'(x)\right),\quad x\in X.
\]
By \cite[Lemma~1.2.23]{FaPhD} and Remark \ref{rem:localtight} the operator $R(\lambda,A)$, $\lambda\in\rho(A)$, is local, i.e., for each $\varepsilon'>0$ there exists $K'>0$ and $q''\in\semis$ such that
\[
p'(\res(\lambda,A)x)\leq K'q''(x)+\varepsilon'\left\|x\right\|,\quad x\in X.
\]
With $K'':=p(y)MK'$ this leads to the following inequality
\[
p(B\res(\lambda,A)x)\leq K''\left(q''(x)+\frac{1}{M}q'(x)\right)+\varepsilon\left\|x\right\|,\quad x\in X.
\]
where $\varepsilon=Mp(y)\varepsilon'$. We conclude from Remark \ref{rem:SumSeminorm} that $q''+\frac{1}{M}q'\in\semis$, showing that $B\res(\lambda,A)$ is local. Since $(A,\dom(A))$ is a Hille--Yosida operator one gets
\[
\left\|B\res(\lambda,A)x\right\|\leq\left|\varphi(\res(\lambda,A)x\right|\cdot\left\|y\right\|\leq\left\|\varphi\right\|\left\|\res(\lambda,A)x\right\|\left\|y\right\|\leq\left\|\varphi\right\|\frac{M}{\left|\lambda-\omega\right|}\left\|x\right\|\left\|y\right\|.
\]
Since all constants are fixed from the beginning our expression becomes smaller than $1$ by chosing $\lambda>\omega$ large enough, this shows that rank-one perturbations fit in our setting, i.e., $A+B$ generates a bi-continuous semigroup by an application of Theorem \ref{thm:BiMVPos}. Observe that the argumentation can be extended to finite rank operators.

\subsection{Gauss--Weierstrass semigroup on $\M(\RR)$}\label{subsec:GWMeasure}

Now we discuss a second example, i.e., we consider the space $\M(\Omega)$ of bounded Borel measures on $\Omega$, where $\Omega$ is a Polish space. 
First of all we show that $\M(\Omega)$ satisfies the conditions of Definition \ref{def:biAL}, i.e., we show that the space $\M(\Omega)$ of bounded Borel measures is a bi-$\AL$-space with respect to an appropiate locally convex topology. As already mentioned in Section \ref{subsec:AdBiCont} every bi-continuous semigroup on $\BC(\Omega)$ with respect to the compact-open topology $\tau_{\mathrm{co}}$ gives rise to a bi-continuous semigroup on $\M(\Omega)$ with respect to $\tau=\sigma(\M(\Omega),\BC(\Omega))$ and vice versa. So the locally convex topology on $\M(\Omega)$ is given by the weak$^*$-topology generated by the family of seminorms given by
\[
\semis:=\left\{p_f:\ f\in\BC(\RR)\right\},
\]
where 
\[
p_f(\mu):=\left|\int_{\Omega}{f\ \dd\mu}\right|,\quad \mu\in\M(\Omega).
\]
The norm on $\M(\Omega)$ is the total variation norm, which turns the space into a Banach space. Moreover, the norm satisfies the first condition of Definition \ref{def:biAL}, i.e., $\left\|\mu+\nu\right\|_{\M(\Omega)}=\left\|\mu\right\|_{\M(\Omega)}+\left\|\nu\right\|_{\M(\Omega)}$. 
What is more, even the second condition of the additivity of the seminorms is fullfilled. In fact, let
\[
\semis_+:=\left\{p_{f}:\ f\in\BC(\Omega),\ f\geq0\right\},
\]
and let $\tau_+$ denote the locally convex topology generated by $\semis_+$. We first observe that for each $p=p_f\in\semis_+$ one has for $\mu,\nu\geq0$
\begin{align*}
p_f(\mu)+p_f(\nu)
&=\left|\int_{\Omega}{f\ \dd{\mu}}\right|+\left|\int_{\Omega}{f\ \dd{\nu}}\right|\\
&=\int_{\Omega}{f\ \dd{\mu}}+\int_{\Omega}{f\ \dd{\nu}}\\
&=\int_{\Omega}{f\ \dd{(\mu+\nu)}}\\
&=\left|\int_{\Omega}{f\ \dd{(\mu+\nu)}}\right|\\
&=p_f(\mu+\nu).
\end{align*}

We claim that $\tau=\tau_+$.

\begin{lemma}
The family $\semis_+$ generates the topology $\tau$, i.e., for a net $(\mu_{\alpha})_{\alpha\in\Lambda}$ in $\M(\Omega)$ and $\mu\in\M(\Omega)$:
\[
\mu_{\alpha}\stackrel{\tau}{\rightarrow}\mu\ \Longleftrightarrow\ \mu_{\alpha}\stackrel{\tau_+}{\rightarrow}\mu
\]
\end{lemma}

\begin{proof}
Suppose that 
\begin{align}\label{eqn:ConvTau}
\int_{\Omega}{f\ \dd\mu_{\alpha}}\rightarrow\int_{\Omega}{f\ \dd{\mu}},
\end{align}
for each $f\in\BC(\Omega)$, then obviously \eqref{eqn:ConvTau} also holds for every positive $f\in\BC(\Omega)$. For the other implication notice that for arbitrary $f\in\BC(\Omega)$ one can decompose $f$ into a positive and negative part, i.e., there exist positive $f_+,f_-\in\BC(\Omega)$ such that
\[
f=f_+-f_-.
\]
Then
\[
\int_{\Omega}{f_+\ \dd\mu_{\alpha}}\rightarrow\int_{\Omega}{f_+\ \dd{\mu}},
\]
and
\[
\int_{\Omega}{f_-\ \dd\mu_{\alpha}}\rightarrow\int_{\Omega}{f_-\ \dd{\mu}}.
\]
Since
\[
\int_{\Omega}{f\ \dd\mu_{\alpha}}=\int_{\Omega}{f_+\ \dd\mu_{\alpha}}-\int_{\Omega}{f_-\ \dd{\mu_{\alpha}}},
\]
we observe that
\[
\int_{\Omega}{f\ \dd\mu_{\alpha}}\rightarrow\int_{\Omega}{f_+\ \dd{\mu}}-\int_{\Omega}{f_-\ \dd{\mu}},
\]
and hence we are done.
\end{proof}

We now consider the special case $\Omega=\RR$, i.e., we consider the space $\M(\RR)$ of bounded Borel measures on $\RR$ with respect to the Borel $\sigma$-algebra $\BBB(\RR)$. Notice that since we consider bounded Borel measures on the separable metric space $\RR$ every $\mu\in\M(\RR)$ is also a Radon measure. 

We first recall the definition of the (centered) Gaussian measure $\gamma_t$, $t>0$, on $\RR$ defined by
\[
\gamma_t(C)=\frac{1}{\sqrt{2\pi t}}\int_{C}{\ee^{-\frac{\left|x\right|^2}{2t}}\ \dd{\lambda^1}},\quad C\in\BBB(\RR),
\]
where $\lambda^1$ denotes the Lebesgue measure on $\RR$. This measure $\gamma_t$, $t>0$, is a strictly positive bounded Borel measure and in particular a Radon measure. As a matter of fact, $\gamma_t$, $t>0$, is absolutely continuous with respect to the Lebesgue measure with density $\varphi_t$ given by
\[
\varphi_t(x)=\frac{1}{\sqrt{2\pi t}}\ee^{-\frac{\left|x\right|^2}{2t}}.
\]
Next we recall the convolution of measures. For $\mu,\nu\in\M(\RR)$ one defines the convolution of $\mu$ and $\nu$ by
\[
(\mu\ast\nu)(\Omega)=\int_{\RR}{\int_{\RR}{\textbf{1}_{\Omega}(x+y)\ \dd\mu(x)\ \dd\nu(y)}}=\int_{\RR}{\nu(\Omega-x)\ \dd\mu(x)}.
\]

\begin{remark}\label{rem:density}
If one of the measures $\mu,\nu\in\M(\RR)$ has some density with respect to the Lebesgue measure, say $\nu=g\cdot\lambda^1$ for some non-negative, integrable function relative to the Lebesgue measure, as it is the case for $\gamma_t$, then by an application of Tonelli's theorem and the translation invariance of the Lebesgue measure, the convolution has density $g\ast\mu$ defined by
\[
(g\ast\mu)(x)=\int_{\RR}{g(x-y)\ \dd\mu(y)}.
\]
\end{remark}

Now we define a family of operators $(T(t))_{t\geq0}$ on $\M(\RR)$ by $T(0)\mu=\mu$ and 
\begin{align}\label{eqn:semig}
T(t)\mu=\gamma_t\ast\mu,\quad t>0.
\end{align}

Since $\gamma_t\geq0$ for each $t>0$ we conclude that the family $(T(t))_{t\geq0}$ consists of positive operators on $\M(\RR)$. The next result shows that $(T(t))_{t\geq0}$ is in fact a bi-continuous semigroup on $\M(\RR)$ equipped with the total variation norm and the topology $\sigma(\M(\RR),\BC(\Omega))$. For that recall that the so-called Gauss--Weierstrass semigroup on $\BC(\RR)$, here denoted by $(T_*(t))_{t\geq0}$, is defined by $T(0)f:=f$ and
\[
T_*(t)f:=\varphi_t\ast f,\quad t>0.
\]
As a matter of fact, $(T_*(t))_{t\geq0}$ is a bi-continuous semigroup on $\BC(\RR)$ with respect to the compact-opem topology, cf. \cite[Sect.~2.4]{FaPhD}. We now relate our semigroup $(T(t))_{t\geq0}$ to $(T_*(t))_{t\geq0}$.

\begin{theorem}
The semigroup $(T(t))_{t\geq0}$ is the adjoint of the Gauss--Weierstrass semigroup $(T_*(t))_{t\geq0}$ on $\BC(\RR)$. 
\end{theorem}

\begin{proof}
We observe that by Remark \ref{rem:density} for $f\in\BC(\RR)$ with $f\geq0$ and $\mu\in\M(\RR)$ the following holds.
\begin{align*}
\left\langle f,T(t)\mu\right\rangle&=\left|\int_{\RR}{f(x)\ \dd(T(t)\mu)(x)}\right|\\
&=\left|\int_{\RR}{f(x)\ \dd(\gamma_t\ast\mu)(x)}\right|\\
&=\left|\int_{\RR}{f(x)\ \dd((\varphi_t\cdot\lambda^1)\ast\mu)}(x)\right|\\
&=\left|\int_{\RR}{f(x)\ \dd((\varphi_t\ast\mu)\cdot\lambda^1)(x)}\right|\\
&=\left|\int_{\RR}{f(x)(\varphi_t\ast\mu)(x)\ \dd\lambda^1(x)}\right|\\
&=\left|\int_{\RR}{f(x)}\int_{\RR}{\varphi_t(x-y)\ \dd\mu(y)}\ \dd\lambda^1(x)\right|\\
&=\left|\int_{\RR}(f\ast\varphi_t)(y)\ \dd\mu(y)\right|\\
&=\left|\int_{\RR}{{T}_*(t)f(y)\ \dd\mu(y)}\right|\\
&=\left\langle {T}_*(t)f,\mu\right\rangle.
\end{align*}
Here $\left\langle \cdot,\cdot\right\rangle$ denotes the pairing between $\BC(\RR)$ and $\M(\RR)$. This proves the assertion.
\end{proof}

\begin{remark}\label{rem:PropGaussWei}
\begin{abc}
	\item The notation $(T_*(t))_{t\geq0}$ is intuitive since $(T_*(t))_{t\geq0}$ is the preadjoint of $(T(t))_{t\geq0}$.
	\item Since $({T}_*(t))_{t\geq0}$ is known to be bi-continuous on $\BC(\RR)$ with respect to the compact-open topology we conclude by Theorem \ref{thm:DualCbM} that $(T(t))_{t\geq0}$ on $\M(\RR)$ is bi-continuous with respect to $\sigma(\M(\RR),\BC(\Omega))$.
	\item Recall that the generator $(A_*,\dom(A_*))$ of the Gauss--Weierstrass semigroup on $\BC(\RR)$ is given by $(\Delta,\BC^2(\RR))$ where $\Delta$ denotes the Laplacian and $\BC^2(\RR)$ the space of twice continuous differentiable functions with bounded derivatives, cf. \cite{LB} or \cite{Lunardi}
	\item We also notice that $\mathbb{C}_+\subseteq\rho(A_*)$, where $\mathbb{C}_+$ denotes the right-half plane, i.e., all complex numbers with positive real part.
\end{abc}
\end{remark}

The next step is to determine the generator $(A,\dom(A))$ of $(T(t))_{t\geq0}$. To do so we have to consider differentiable measures. The original study of such measures is due to Fomin \cite{Fomin1968} and Skorohod \cite{Sk1957}, \cite[Chapter 4, Sect.~21]{Sk1974}. For more details, we refer to the work of Bogachev \cite{Bog2010}, where differentiable measures are the main subject.

\begin{definition}
A Borel measure $\mu\in\M(\RR)$ is called \emph{Skorohod differentiable} or \emph{S-differentiable} if, for every function $f\in\BC(\RR)$, the function
\[
t\mapsto\int_\RR{f(x-t)\ \dd\mu(x)},
\]
is differentiable.
\end{definition}

The following theorem \cite[Thm.~3.6.1]{Bog2010} shows that the previous definition is equivalent to the existence of a Borel measure $\nu$, called the Skorohod derivative of $\mu$, such that for each bounded continuous function on $\RR$ one has

\begin{align}\label{eqn:SkDiff}
\lim_{t\to0}{\int_\RR{\frac{f(x-t)-f(x)}{t}}\ \dd\mu(x)}=\int_\RR{f(x)\ \dd\nu(x)}.
\end{align}

\begin{theorem}\label{thm:SdiffMeas}
Let $\mu$ be a Skorohod differentiable Borel measure on $\RR$. Then there exists a Borel measure $\nu$ which is its  Skorohod derivative, i.e., $\nu$ satisfies \eqref{eqn:SkDiff} for all bounded continuous functions $f$ on $\RR$.
\end{theorem}

\begin{remark}\label{rem:SkoDiff}
\begin{abc}
	\item The Skorohod derivative $\nu$ of $\mu$ as it appears in \eqref{eqn:SkDiff} and Theorem \ref{thm:SdiffMeas} is denoted by $\mathscr{D}\mu$, i.e., $\nu=\mathscr{D}\mu$.
	\item By \cite[Prop.~3.4.1]{Bog2010} one has that for a bounded Borel measure $\mu$ on $\RR$ the measure $\mu$ is Skorohod differentiable if and only if it has density of bounded variation, in particular every Skorohod differentiable measure on $\RR$ admits a bounded density. In that case the density of the Skorohod derivative is just the (distributional) derivative of the original density.
	\item For higher order derivatives we recall from \cite[Prop.~3.7.1]{Bog2010} that if the map $t\mapsto\int_{\RR}{f(x-t)\ \dd\mu(x)}$ is $n$-times differentiable, then $\mu$ is $n$-times Skorohod differentiable, i.e., for all $f\in\BC(\RR)$ the function
	\[
	(t_1,\ldots,t_n)\mapsto\int_{\RR}{f(x+t_1+\cdots+t_n)\ \dd\mu(x)},
	\]
	has partial derivatives $\partial_{t_1}\cdots\partial_{t_n}$.
	\item $\mu$ is Skorohod differentiable with Skorohod derivative $\mathscr{D}\mu$ if and only if for each $f\in\BC^{\infty}(\RR)$
	\[
	\int_{\RR}{\frac{\dd}{\dd{x}}f(x)\ \dd\mu(x)}=-\int_{\RR}{f(x)\ \dd(\mathscr{D}\mu)(x)},
	\]
	in particular one concludes that Skorohod differentiability of a measure $\mu$ is equivalent to that its derivative in the sense of distributions is a bounded measure, cf. \cite[Prop.~3.4.3]{Bog2010}.
	\end{abc}
\end{remark}



Now we prove the following result.

\begin{theorem}
The generator $(A,\dom(A))$ of $(T(t))_{t\geq0}$ is given by
\[
A\mu:=\Delta\mu,\quad \dom(A):=\left\{\mu\in\M(\RR): \mu\ \text{is twice Skorohod differentiable}\ \right\},
\]
where $\Delta\mu$ denotes the second order Skorohod derivative of $\mu$ introduced in Remark \ref{rem:SkoDiff}$\mathrm{(c)}$.
\end{theorem}

\begin{proof}
Let us denote the generator of our semigroup $(T(t))_{t\geq0}$ by $(C,\dom(C))$. By Lemma \ref{lem:DomAdBiContSemi} and Remark \ref{rem:PropGaussWei} we conclude that the domain of the generator is of the following form
\[
\dom(C)=\left\{\mu\in\M(\RR):\ \exists\nu\in\M(\RR)\ \forall f\in\BC^2(\RR):\ \left\langle f,\nu\right\rangle=\left\langle \Delta f,\mu\right\rangle\right\}.
\]
Now let $\mu\in\M(\RR)$ be twice Skorohod differentiable and denote its second derivative by $\nu:=\Delta\mu$. For $f\in\BC^2(\RR)$ one has
\[
\left\langle \Delta f,\mu\right\rangle=\int_\RR{\Delta f\ \dd\mu}=\int_{\RR}{f\ \dd\nu}=\left\langle f,\nu\right\rangle,
\]
showing that $\dom(A)\subseteq\dom(C)$, cf. Lemma \ref{lem:DomAdBiContSemi}. For the converse let $\mu\in\dom(C)$, i.e., there exists $\nu\in\M(\RR)$ such that for each $f\in\BC^2(\RR)$ holds that
\[
\left\langle f,\nu\right\rangle=\left\langle \Delta f,\mu\right\rangle=\left\langle \lim_{t\to0}{\frac{T_*(t)f-f}{t}},\mu\right\rangle=\left\langle f,\lim_{t\to0}{\frac{T(t)\mu-\mu}{t}}\right\rangle.
\]
From this we conclude that $\sigma(\M(\RR),\BC(\Omega))-\lim_{t\to0}{\frac{T(t)\mu-\mu}{t}}$ exists. Combining this with \cite[Thm.~3.6.4]{Bog2010} we obtain that $\mu\in\dom(A)$. This finishes the proof.
\end{proof}

\begin{remark}
In Remark \ref{rem:PropGaussWei} we observed that $\mathbb{C}_+\subseteq\rho(A)$ and since $\sigma(A_*)=\sigma(A)$ by \cite[Chapter II, Sect.~2.5]{EN}, we also obtain that $\mathbb{C}_+\subseteq\rho(A)$, especially $\lambda\in\rho(A)$ for $\lambda>0$. Moreover we observe that $(T(t))_{t\geq0}$ is a bounded semigroup and hence we also conclude that $\mathbb{C}_+\subseteq\rho(A)$.
\end{remark}

Now fix a positive, Lebesgue integrable (and unbounded) function $\psi:\RR\to\RR$ and define $B:\dom(A)\to\M(\RR)$ by 
\[
B\mu:=\psi\cdot\mu,\quad \mu\in\dom(A).
\]        
Observe that indeed $B\mu\in\M(\RR)$ by an application of Remark \ref{rem:SkoDiff}. To show that the operator $B$ satisfies the conditions of Theorem \ref{thm:BiMVPos} let $\lambda\in\rho(A)$, especially we can choose $\lambda\in\RR$ with $\lambda>0$ and observe that by the duality between $\BC(\RR)$ and $\M(\RR)$
\[
p_f(B\res(\lambda,A)\mu)=\left|\left\langle f,B\res(\lambda,A)\mu\right\rangle\right|=\left|\left\langle f\psi ,\res(\lambda,A)\mu\right\rangle\right|=\left|\left\langle \res(\lambda,A_*)(f\psi),\mu\right\rangle\right|
\]
for all $f\in\BC(\RR)$ and $\mu\in\dom(A)$. Moreover, by an application of Fubini's theorem and an explicit calculation of an integral one obtains
\begin{align*}
\left|\res(\lambda,A_*)(f\psi)(x)\right|&=\left|\int_0^{\infty}{\int_\RR{\ee^{-\lambda t}\varphi(x-y)f(y)\psi(y)\ \dd{y}\ \dd{t}}}\right|\\
&=\left|\int_\RR{\int_0^{\infty}{\ee^{-\lambda t}\varphi(x-y)f(y)\psi(y)\ \dd{t}\ \dd{y}}}\right|\\
&=\left|\int_\RR{\frac{1}{\sqrt{2\lambda}}\ee^{-\sqrt{2\lambda}\left|x-y\right|}}f(y)\psi(y)\ \dd{y}\right|\\
&=\left|(\xi_{\lambda}\ast f\psi)(x)\right|,
\end{align*}
where $\xi_{\lambda}(x):=\frac{1}{\sqrt{2\lambda}}\ee^{-\sqrt{2\lambda}\left|x\right|}$ is continuous and hence the convolution $\xi_{\lambda}\ast(f\psi)$ is continuous. Furthermore, by Young's convolution inequality and the assumption that $\psi$ is integrable we obtain
\[
\left\|\xi_{\lambda}\ast f\psi\right\|_{\infty}\leq\left\|\xi_{\lambda}\right\|_{\infty}\left\|f\psi\right\|_1\leq\left\|\xi_{\lambda}\right\|_{\infty}\left\|f\right\|_{\infty}\left\|\psi\right\|_1<\infty
\]
hence $h:=\xi_{\lambda}\ast f\psi\in\BC(\RR)$ and $p_f(B\res(\lambda,A)\mu)=p_h(\mu)$. By using that $\M(\RR)$ is the dual of $\BC(\RR)$ we obtain the following norm estimate
\[
\left\|B\res(\lambda,A)\mu\right\|=\sup_{\substack{f\in\BC(\RR)\\\left\|f\right\|_{\infty}\leq1}}{\left|\left\langle f,B\res(\lambda,A)\mu\right\rangle\right|}=\sup_{\substack{f\in\BC(\RR)\\\left\|f\right\|_{\infty}\leq1}}{\left|\left\langle \res(\lambda,A_*)f\psi,\mu\right\rangle\right|}\leq\sup_{\substack{f\in\BC(\RR)\\\left\|f\right\|_{\infty}\leq1}}{\left\|\res(\lambda,A_*)f\psi\right\|_{\infty}\left\|\mu\right\|}.
\]
Moreover,
\begin{align*}
\left|\res(\lambda,A_*)(f\psi)(x)\right|&=\left|\int_0^{\infty}\ee^{-\lambda t}{T}_*(t)f(x)\psi(x)\ \dd{t}\right|\\
&=\left|\int_0^{\infty}{\ee^{-\lambda t}\int_{\RR}\varphi_t(x-y)f(y)\psi(y)\ \dd{y} \dd{t}}\right|\\
&\leq\left\|f\right\|_{\infty}\int_0^{\infty}{\ee^{-\lambda t}(\varphi\ast\psi)(x)\ \dd{t}}\\
&\leq\left\|f\right\|_{\infty}\left\|\psi\right\|_1\int_0^{\infty}{\ee^{-\lambda t}\left\|\varphi_t\right\|_{\infty}\ \dd{t}}\\
&=\left\|f\right\|_{\infty}\left\|\psi\right\|_1\int_0^{\infty}{\frac{\ee^{-\lambda t}}{\sqrt{2\pi t}}\ \dd{t}}\\
&=\frac{\left\|f\right\|_{\infty}\left\|\psi\right\|_1}{\sqrt{2\lambda}}
\end{align*}
Hence, for $\lambda$ big enough we obtain $\left\|\res(\lambda,A_*)f\psi\right\|<1$ and hence $\left\|B\res(\lambda,A)\right\|<1$. Now we can apply Theorem \ref{thm:BiMVPos} and conclude that $A+B$ generates a positive bi-continuous semigroup on $\M(\RR)$.

\begin{remark}
Suppose that $f\in\mathrm{L}^p(\RR)$ is unbounded. Observe that there exists $g\in\mathrm{L}^1(\RR)$ and $h\in\mathrm{L}^{\infty}(\RR)$ such that $f=g+h$. In particular, we consider the operator $B:\dom(A)\to\M(\RR)$ defined by $B\mu:=f\cdot\mu$ as above. Since the case for $f\in\mathrm{L}^1(\RR)$ is treated before and $f\in\mathrm{L}^{\infty}(\RR)$ gives rise to a bounded perturbation we conclude that we can extend our previous result to the whole scale of $\mathrm{L}^p$-spaces.
\end{remark}

\section*{Acknowledgement}
\noindent The author would like to thank B\'{a}lint Farkas for the suggestion of the topic of this paper as well as for the fruitful discussions and the continuous support during writing.


\begin{thebibliography}{10}

\bibitem{ABE2014}
M.~Adler, M.~Bombieri, and K.-J. Engel.
\newblock On perturbations of generators of {$C_0$}-semigroups.
\newblock {\em Abstr. Appl. Anal.}, pages Art. ID 213020, 13, 2014.

\bibitem{Arendt1987}
W.~Arendt.
\newblock Resolvent positive operators.
\newblock {\em Proc. London Math. Soc. (3)}, 54(2):321--349, 1987.

\bibitem{Positive1986}
W.~Arendt, A.~Grabosch, G.~Greiner, U.~Groh, H.~P. Lotz, U.~Moustakas,
  R.~Nagel, F.~Neubrander, and U.~Schlotterbeck.
\newblock {\em One-parameter semigroups of positive operators}, volume 1184 of
  {\em Lecture Notes in Mathematics}.
\newblock Springer-Verlag, Berlin, 1986.

\bibitem{ArendtRhandi1991}
W.~Arendt and A.~Rhandi.
\newblock Perturbation of positive semigroups.
\newblock {\em Arch. Math. (Basel)}, 56(2):107--119, 1991.

\bibitem{PertPosAppl2006}
J.~Banasiak and L.~Arlotti.
\newblock {\em Perturbations of positive semigroups with applications}.
\newblock Springer Monographs in Mathematics. Springer-Verlag London, Ltd.,
  London, 2006.

\bibitem{Positive2017}
A.~B\'atkai, M.~Kramar~Fijav\v{z}, and A.~Rhandi.
\newblock {\em Positive operator semigroups}, volume 257 of {\em Operator
  Theory: Advances and Applications}.
\newblock Birkh\"auser/Springer, Cham, 2017.
\newblock From finite to infinite dimensions, With a foreword by Rainer Nagel
  and Ulf Schlotterbeck.

\bibitem{Bog2010}
V.~I. Bogachev.
\newblock {\em Differentiable measures and the {M}alliavin calculus}, volume
  164 of {\em Mathematical Surveys and Monographs}.
\newblock American Mathematical Society, Providence, RI, 2010.

\bibitem{BombieriPhD}
M.~Bombieri.
\newblock {\em On {W}eiss--{S}taffans Perturbations of Semigroup Generators}.
\newblock PhD thesis, Eberhard-Karls-Universit\"at T\"ubingen, 2015.

\bibitem{BuddePhD}
C.~Budde.
\newblock {\em General extrapolation spaces and perturbations of bi-continuous
  semigroups}.
\newblock PhD thesis, Bergische Universit\"at Wuppertal, 2019.

\bibitem{BF2}
C.~Budde and B.~Farkas.
\newblock A {D}esch--{S}chappacher perturbation theorem for bi-continuous
  semigroups.
\newblock to appear, 2019.

\bibitem{BF}
C.~Budde and B.~Farkas.
\newblock Intermediate and extrapolated spaces for bi-continuous operator
  semigroups.
\newblock {\em J. Evol. Equ.}, 19(2):321--359, 2019.

\bibitem{Conway}
J.~B. Conway.
\newblock {\em A course in abstract analysis}, volume 141 of {\em Graduate
  Studies in Mathematics}.
\newblock American Mathematical Society, Providence, RI, 2012.

\bibitem{Cristescu2008}
R.~Cristescu.
\newblock Vector seminorms, spaces with vector norm, and regular operators.
\newblock {\em Rev. Roumaine Math. Pures Appl.}, 53(5-6):407--418, 2008.

\bibitem{DS1984}
W.~Desch and W.~Schappacher.
\newblock On relatively bounded perturbations of linear {$C_0$}-semigroups.
\newblock {\em Ann. Scuola Norm. Sup. Pisa Cl. Sci. (4)}, 11(2):327--341, 1984.

\bibitem{EFHN2015}
T.~Eisner, B.~Farkas, M.~Haase, and R.~Nagel.
\newblock {\em Operator theoretic aspects of ergodic theory}, volume 272 of
  {\em Graduate Texts in Mathematics}.
\newblock Springer, Cham, 2015.

\bibitem{EN}
K.-J. Engel and R.~Nagel.
\newblock {\em One-parameter semigroups for linear evolution equations}, volume
  194 of {\em Graduate Texts in Mathematics}.
\newblock Springer-Verlag, New York, 2000.
\newblock With contributions by S. Brendle, M. Campiti, T. Hahn, G. Metafune,
  G. Nickel, D. Pallara, C. Perazzoli, A. Rhandi, S. Romanelli and R.
  Schnaubelt.

\bibitem{ESF2005}
A.~Es-Sarhir and B.~Farkas.
\newblock Positivity of perturbed {O}rnstein-{U}hlenbeck semigroups on
  {$C_b(H)$}.
\newblock {\em Semigroup Forum}, 70(2):208--224, 2005.

\bibitem{FaPhD}
B.~Farkas.
\newblock {\em Perturbations of Bi-Continuous Semigroups}.
\newblock PhD thesis, E\"otv\"os Lor\'and University, 2003.

\bibitem{FaStud}
B.~Farkas.
\newblock Perturbations of bi-continuous semigroups.
\newblock {\em Studia Math.}, 161(2):147--161, 2004.

\bibitem{FaSF}
B.~Farkas.
\newblock Perturbations of bi-continuous semigroups with applications to
  transition semigroups on {$C_b(H)$}.
\newblock {\em Semigroup Forum}, 68(1):87--107, 2004.

\bibitem{Farkas2011}
B.~Farkas.
\newblock Adjoint bi-continuous semigroups and semigroups on the space of
  measures.
\newblock {\em Czechoslovak Math. J.}, 61(136)(2):309--322, 2011.

\bibitem{FL2009}
B.~Farkas and L.~Lorenzi.
\newblock On a class of hypoelliptic operators with unbounded coefficients in
  {$\mathbb{R}^N$}.
\newblock {\em Communications on Pure \& Applied Analysis}, 8:1159, 2009.

\bibitem{Fomin1968}
S.~V. Fomin.
\newblock Differentiable measures in linear spaces.
\newblock {\em Us\.{p}ehi Mat. Nauk}, 23(1 (139)):221--222, 1968.

\bibitem{KuPHD}
F.~K\"uhnemund.
\newblock {\em Bi-continuous semigroups on spaces with two topologies: Theory
  and applications}.
\newblock PhD thesis, Eberhard-Karls-Universi\"at T\"ubingen, 2001.

\bibitem{Ku}
F.~K\"uhnemund.
\newblock A {H}ille--{Y}osida theorem for bi-continuous semigroups.
\newblock {\em Semigroup Forum}, 67(2):205--225, 2003.

\bibitem{Kuehner2019}
V.~K{\"u}hner.
\newblock What can koopmanism do for attractors in dynamical systems?
\newblock {\em The Journal of Analysis}, Nov 2019.

\bibitem{LB}
L.~Lorenzi and M.~Bertoldi.
\newblock {\em Analytical methods for {M}arkov semigroups}, volume 283 of {\em
  Pure and Applied Mathematics (Boca Raton)}.
\newblock Chapman \& Hall/CRC, Boca Raton, FL, 2007.

\bibitem{Lunardi}
A.~Lunardi.
\newblock {\em Interpolation theory}.
\newblock Appunti. Scuola Normale Superiore di Pisa (Nuova Serie). [Lecture
  Notes. Scuola Normale Superiore di Pisa (New Series)]. Edizioni della
  Normale, Pisa, second edition, 2009.

\bibitem{PMN1991}
P.~{Meyer-Nieberg}.
\newblock {\em {Banach lattices.}}
\newblock Berlin etc.: Springer-Verlag, 1991.

\bibitem{Rudin}
W.~Rudin.
\newblock {\em Functional analysis}.
\newblock International Series in Pure and Applied Mathematics. McGraw-Hill,
  Inc., New York, second edition, 1991.

\bibitem{Schaefer1971}
H.~H. Schaefer.
\newblock {\em Topological vector spaces}.
\newblock Springer-Verlag, New York-Berlin, 1971.
\newblock Third printing corrected, Graduate Texts in Mathematics, Vol. 3.

\bibitem{Schaefer1974}
H.~H. Schaefer.
\newblock {\em Banach lattices and positive operators}.
\newblock Springer-Verlag, New York-Heidelberg, 1974.
\newblock Die Grundlehren der mathematischen Wissenschaften, Band 215.

\bibitem{Sci2016}
A.~K. Scirrat.
\newblock {\em Evolution Semigroups for Well-Posed, Non-Autonomous Evolution
  Families}.
\newblock PhD thesis, Louisiana State University and Agricultural and
  Mechanical College, 2016.

\bibitem{Sk1957}
A.~V. Skorohod.
\newblock On the differentiability of measures which correspond to stochastic
  processes. {I}. {P}rocesses with independent increments.
\newblock {\em Teor. Veroyatnost. i Primenen}, 2:417--443, 1957.

\bibitem{Sk1974}
A.~V. Skorohod.
\newblock {\em Integration in {H}ilbert space}.
\newblock Springer-Verlag, New York-Heidelberg, 1974.
\newblock Translated from the Russian by Kenneth Wickwire, Ergebnisse der
  Mathematik und ihrer Grenzgebiete, Band 79.

\bibitem{van1992adjoint}
J.~van Neerven.
\newblock {\em The adjoint of a semigroup of linear operators}, volume 1529 of
  {\em Lecture Notes in Mathematics}.
\newblock Springer-Verlag, Berlin, 1992.

\bibitem{Voigt1977}
J.~Voigt.
\newblock On the perturbation theory for strongly continuous semigroups.
\newblock {\em Math. Ann.}, 229(2):163--171, 1977.

\bibitem{Voigt1989}
J.~Voigt.
\newblock On resolvent positive operators and positive {$C_0$}-semigroups on
  {$AL$}-spaces.
\newblock {\em Semigroup Forum}, 38(2):263--266, 1989.
\newblock Semigroups and differential operators (Oberwolfach, 1988).

\end{thebibliography}

\end{document}